\documentclass[10pt]{amsart}
\usepackage{latexsym, amsmath,amssymb}

\usepackage{graphicx}
\usepackage{color} 

\usepackage[colorlinks=true]{hyperref}
\hypersetup{urlcolor=blue, citecolor=blue}

\newcommand{\cal}{\mathcal}

\renewcommand{\epsilon}{\varepsilon}
\newcommand{\newsection}[1]
{\subsection{#1}\setcounter{theorem}{0} \setcounter{equation}{0}
\par\noindent}

\newtheorem{theorem}{Theorem}

\newtheorem{lemma}[theorem]{Lemma}
\newtheorem{corr}[theorem]{Corollary}

\newtheorem{proposition}[theorem]{Proposition}
\newtheorem{deff}[theorem]{Definition}

\newcommand{\bth}{\begin{theorem}}
\newcommand{\ble}{\begin{lemma}}
\newcommand{\bcor}{\begin{corr}}

\newcommand{\bdeff}{\begin{deff}}

\newcommand{\bprop}{\begin{proposition}}
\newcommand{\ele}{\end{lemma}}
\newcommand{\ecor}{\end{corr}}
\newcommand{\edeff}{\end{deff}}

\newcommand{\eprop}{\end{proposition}}

\newcommand{\la}{\lambda}

\newcommand{\e}{\varepsilon}

\newcommand{\supp}{\text{supp }}
\renewcommand{\Pi}{\varPi}

\renewcommand{\epsilon}{\varepsilon}

\newcommand{\R}{{\mathbb R}}

\newcommand{\Zt}{{\mathbb Z}^2}
\newcommand{\torus}{{\mathbb T}^2}

\newcommand{\vertiii}[1]{{\left\vert\kern-0.25ex\left\vert\kern-0.25ex\left\vert #1 
    \right\vert\kern-0.25ex\right\vert\kern-0.25ex\right\vert}}

\begin{document}

\title[Refined and microlocal Kakeya-Nikodym bounds]
{Refined and microlocal Kakeya-Nikodym bounds for eigenfunctions in two dimensions}
%\thanks{The authors were supported in part by the NSF}
%
%
%
%
%
%
\thanks{The authors were supported in part by the NSF grants DMS-1301717 and DMS-1361476, respectively.}

\author[M. D. Blair]{Matthew D. Blair}
\address{Department of Mathematics and Statistics, University of New Mexico, Albuquerque, NM 87131, USA}
\email{blair@math.unm.edu}
\author[C. D. Sogge]{Christopher D. Sogge}
\address{Department of Mathematics, Johns Hopkins University, Baltimore, MD 21093, USA}
\email{sogge@jhu.edu}

\begin{abstract}
We obtain some improved essentially sharp Kakeya-Nikodym estimates for eigenfunctions in two-dimensions.  We obtain these by proving stronger related
microlocal estimates involving a natural decomposition of phase space that is
adapted to the geodesic flow.
\end{abstract}

\maketitle

\newsection{Introduction and main results}

Suppose that $(M,g)$ is a two-dimensional compact Riemannian manifold
and $\{e_\la\}$ are the associated eigenfunctions.  That is, if
$\Delta_g$ is the Laplace-Beltrami operator, we have
$$-\Delta_g e_\la(x)=\la^2e_\la(x),$$
and, we assume throughout, that the eigenfunctions are normalized to have
$L^2$-norm one, i.e.,
$$\int_M |e_\la|^2\, dV_g=1,$$
where $dV_g$ is the volume element.

The purpose of this paper is to obtain essentially sharp estimates that link, in
two dimensions, the size of $L^p$-norms of eigenfunctions with $2<p<6$ to their
$L^2$-concentration near geodesics.  
Specifically, we have the following:

\begin{theorem}\label{theorem1.1}  For every $0<\e_0\le \frac12$ we have
\begin{equation}\label{1.1}
\|e_\la\|_{L^4(M)}\lesssim_{\e_0} \la^{\frac{\e_0}4}
\|e_\la\|_{L^2(M)}^{\frac12}
\times
\vertiii{e_\la}_{KN(\la,\e_0)}^{\frac12}
\end{equation}
if
\begin{equation}\label{1.1'}
\vertiii{e_\la}_{KN(\la,\e_0)}=
\Bigl(\, \sup_{\gamma\in \varPi} \, \la^{\frac12-\e_0} \int_{{\mathcal T}_{\la^{-\frac12+\e_0}}(\gamma)}|e_\la|^2
\, dV\Bigr)^{\frac12}
\end{equation}
Equivalently, if $\e_0>0$ then there is a $C=C(\e_0,M)$ so that
\begin{equation}\label{1.1''}
\|e_\la\|_{L^4}\le C\la^{\frac18}\|e_\la\|_{L^2(M)}^{\frac12}\times \Bigl(\,  \sup_{\gamma\in \varPi} \,  \int_{{\mathcal T}_{\la^{-\frac12+\e_0}}(\gamma)}|e_\la|^2
\, dV\Bigr)^{\frac14},
\end{equation}
and therefore if $\int_M|e_\la|^2\, dV=1$, we have for any $\e>0$, there is a $C=C(\e,M)$ so that
\begin{equation}\label{1.1'''}
\|e_\la\|_{L^4(M)}\le C\la^{\frac18+\e}\sup_{\gamma\in \varPi}\|e_\la\|_{L^2({\mathcal T}_{\la^{-\frac12}}(\gamma))}^{\frac12}
\le C\la^{\frac1{16}+\e}\sup_{\gamma\in \varPi}\|e_\la\|_{L^4({\mathcal T}_{\la_{j_k}^{-\frac12}}(\gamma))}^{\frac12}.
\end{equation}
\end{theorem}

Here, $\varPi$ denotes the space of unit-length geodesics in $M$ and the last factor in \eqref{1.1'} involves  averages of $|e_\la|^2$ over  $\la^{-\frac12+\e_0}$ tubes about $\gamma\in \varPi$.  Also, for simplicity, we are only stating
things here and throughout for eigenfunctions, but the results easily extend to quasi-modes using results from
\cite{SZq}.

Note that if $\e_0=\frac12$, then \eqref{1.1} is equivalent to the eigenfunction estimates from \cite{SEig}
$$\|e_\la\|_{L^4(M)}\lesssim \la^{\frac18}\|e_\la\|_{L^2(M)},
$$
which are saturated by highest weight spherical harmonics on the standard two-sphere.  We also remark that, up to the factor $\la^{\frac{\e_0}4}$, the estimate \eqref{1.1} is saturated by both the highest weight spherical harmonics and zonal functions on $S^2$.

 Bourgain \cite{BRest} (with a slight loss) and the
second author \cite{SKN} proved this type of inequality where the first norm in the right is raised to the $\frac34$ power and the second to the $\frac14$ power.  The
inequalities in \cite{SKN} were not formulated in this way but easily lead
to this result.
The approach in this paper made an inefficient use of the Cauchy-Schwarz inequality to handle the ``easy'' term (not the bilinear one), which led to this loss.  The strategy for proving \eqref{1.1} will be to make a angular dyadic decomposition of a bilinear expression and pay close attention to the dependence of the bilinear estimates in terms of the angles, which we shall exploit using a multi-layered microlocal decomposition of phase
space.

Before turning to the details of the proof, let us record a few simple corollaries of
our main estimate.

If $\{a_{\la_{j_k}}\}_{k=0}^\infty$ is a sequence depending on a subsequence $\{\la_{j_k}\}$ of the eigenvalues of $\Delta_g$, then we say that
$$a_\la=o_-(\la^\sigma)$$
if there is some $\e>0$ and $C<\infty$ such that
$$|a_\la|\le C(1+\la)^{\sigma-\e}.$$
Then, using the above theorem we get the following:

\begin{corr}
The following are equivalent
\begin{equation}\label{1.2}
\|e_{\la_{j_k}}\|_{L^4(M)}=o_-(\la^{\frac18}_{j_k})
\end{equation}
\begin{equation}\label{1.3}
\sup_{\gamma\in \varPi}
\|e_{\la_{j_k}}\|_{L^4({\mathcal T}_{\la_{j_k}^{-\frac12}}(\gamma))}=o_-(\la^{\frac18}_{j_k})
\end{equation}
\begin{equation}\label{1.4}
\sup_{\gamma\in \varPi}
\|e_{\la_{j_k}}\|_{L^2({\mathcal T}_{\la_{j_k}^{-\frac12}}(\gamma))}=o_-(1).
\end{equation}
%Similarly, the following are equivalent
Also, if either
\begin{equation}\label{1.6}
\sup_{\gamma\in \varPi}
%\|e_{\la_{j_k}}\|_{L^4({\mathcal T}_{\la_{j_k}^{-\frac12}}(\gamma))}=O(\la^{\e}_{j_k}), \quad \forall \e >0
\int_\gamma |e_\la|^2 \, ds=O(\la^{\e}_{j_k}), \quad \forall \e>0,
\end{equation}
or
\begin{equation}\label{1.7}
\sup_{\gamma\in \varPi}
\|e_{\la_{j_k}}\|_{L^2({\mathcal T}_{\la_{j_k}^{-\frac12}}(\gamma))}=O(\la_{j_k}^{-\frac14+\e}), \quad \forall \e>0.
\end{equation}
then
\begin{equation}\label{1.5}
\|e_{\la_{j_k}}\|_{L^4(M)}=O(\la^{\e}_{j_k}), \quad \forall \e>0.
\end{equation}
Here, $ds$ denotes arclength measure on $\gamma$.
\end{corr}

Clearly \eqref{1.2} implies \eqref{1.3}.  Also, \eqref{1.4} follows from \eqref{1.3} and H\"older's inequality.  Since \eqref{1.1} shows that \eqref{1.4}
implies \eqref{1.2}.  The last part of the corollary is also an
easy consequence of the Theorem.
% and since the last statement clearly implies \eqref{1.7} the proof is complete.

Note also that \eqref{1.1'''} says that if $e_{\la_{j_k}}$ is a sequence of eigenfunctions with 
$$\|e_{\la_{j_k}}\|_{L^4(M)}=\Omega(\la_{j_k}^{\frac18})$$ 
then for any $\e$ there must be a sequence of 
shrinking geodesic tubes 
$\{{\mathcal T}_{\la_{j_k}^{-\frac12}}(\gamma_k)\}$ for which for some $c=c_\e>0$ we have%\footnote{The same sequence can be used for both norms by the proof of \eqref{1.1'''} using H\"older's inequality.} 
$$\|e_{\la_{j_k}}\|_{L^4({\mathcal T}_{\la_{j_k}^{-\frac12}}(\gamma_k))}\ge c \,  \la_{j_k}^{\frac18-\e}
%=\lim_{k\to\infty}\la_{j_k}^{\frac{\e_0}4}\|e_{\la_{j_k}}\|_{L^2({\mathcal T}_{\la_{j_k}^{-\frac12}}(\gamma_k))}
$$
%This also follows from the fact that, by \eqref{1.1''} and H\"older's inequality, we have,
%\begin{equation}\label{1.8}
%\|e_\la\|_{L^4(M)}\lesssim_{\e_0} \la^{\frac{\e_0}4}
%\|e_\la\|_{L^2(M)}^{\frac12}
%\times
% \sup_{\gamma\in \varPi}
%\|e_\la\|_{L^4(
%{\mathcal T}_{\la^{-\frac12}}(\gamma))}
%^{\frac12}.
%\end{equation}
In other words, up to a factor of $\la^{-\e}$ for any $\e>0$, they fit the profile of the highest weight spherical harmonics by having  maximal $L^4$-mass
on a sequence of shrinking $\la^{-\frac12}$ tubes.

Like in Bourgain's estimate, \eqref{1.1} involves a slight loss, but this is not so important in view of the above application.  In a later work we hope to show
that \eqref{1.1} holds without this loss (in other words with
$\e_0=0$), which should mainly involve refining
the $S_{1/2,1/2}$ microlocal arguments that are to follow.  Note that, because
of the zonal functions on $S^2$, this result would be sharp.

This paper is organized as follows.  In the next section we shall introduce
a microlocal Kakeya-Nikodym norm and an inequality involving it, \eqref{m.1}, which
implies \eqref{1.1}.  This norm is associated to a decomposition of
phase space which is naturally associated to the geodesic flow on
the cosphere bundle.  In particular each term in the decomposition
will involve bump functions which are supported in tubular neighborhoods
of unit geodesics in $S^*M$.  This decomposition and the resulting
square function arguments are similar to the earlier ones in the joint
paper of Mockenhaupt, Seeger and the second author \cite{MSS}, but there
are some differences and new technical issues that must be overcome.  We
do this and prove our microlocal Kakeya-Nikodym estimate in \S3.  There
after some pseudo-differential arguments we reduce matters to a 
oscillatory integral estimate which is a technical variation on the classical one in H\"ormander~\cite{H}, which was the main step in
his proof of the Carleson-Sj\"olin theorem~\cite{CS}.  The result which
we need does not directly follow from the results in \cite{H}; however,
we can prove it by adapting H\"ormander's argument and using Gauss' lemma.
After doing this, in \S4 we shall see how our results are in some sense
related to Zygmund's theorem \cite{Zy}
saying that in 2-dimensions eigenfunctions on the standard torus have bounded $L^4$-norms.  Specifically we shall see there that if we could obtain
the endpoint version of \eqref{1.1}, we would be able to recover Zygmund's theorem with no loss if we also knew a conjectured result that arcs on
$\la S^1$ of length $\la^{\frac12}$ contain a uniformly bounded number of
%lattice points.  We shall also show in \S4 that we can give a simple
%proof of a recent theorem of the last two authors from \cite{SZ}, and we shall
%also reformulate our Kakeya-Nikodym estimates in terms of averages over
%thin tubes or thin microlocal tubes about geodesics of length
%$T\gg 1$.

In a later paper with S. Zelditch we hope to strengthen our results and also extend
them to higher dimensions, as well as to present applications in the spirit
of \cite{SZ} of the microlocal bounds which we obtain.  The current authors would
like to thank S. Zelditch for a number of stimulating discussions.

%\bigskip

\newsection{Microlocal Kakeya-Nikodym norms}

As in \cite{SKN}, \cite[\S 5.1]{SBook}, we use the fact that we can use a reproducing operator to
write $e_\la =\chi_\la f=\rho(\la-\sqrt{\Delta_g})e_\la$, for $\rho\in {\cal S}$
satisfying $\rho(0)=1$, where, if $\supp \Hat \rho\subset (1,2)$, we also have
modulo $O(\la^{-N})$ errors (see \cite[Lemma 5.1.3]{SBook},
\begin{equation}\label{2.1}
 \chi_\la f(x) =
 \frac1{2\pi}\int \Hat \rho(t) e^{i\la t} \bigl(e^{- it\sqrt{\Delta_g}} f \bigr)(x) \, dt
 %\\
 = \la^{\frac12} \int e^{ i\la \psi(x,y)} a_{\la}(x,y) 
f(y)\, dV(y),
\end{equation}
where 
\begin{equation}\label{2.5}
\psi(x,y)=d_g(x,y)
\end{equation}
is the Riemannian distance function and if, as we may, we assume that the
injectivity radius is 10 or more $a_{\la}$ belongs to a bounded
subset of $C^\infty$ and satisfies
\begin{equation}\label{2.6}
a_{\la}(x,y)=0, \quad \text{if } \, d_g(x,y)\notin (1,2).
\end{equation}

Thus, in order to prove \eqref{1.1}, it suffices to work in a local coordinate
patch and show that if $a$  is smooth and satisfies the support assumptions in  \eqref{2.6} and $0<\delta<1/10$ is small but fixed and if 
$$x_0=(0,y_0), \quad 1/2<y_0<4,$$ 
is also fixed then
\begin{multline}\label{2.7}
\Bigl\|\la^{\frac12} \int e^{i\la \psi(x,y)}a(x,y) f(y)\, dy\Bigr\|_{L^4(B(0,\delta))}^2
\\
\lesssim_{\e_0} \la^{\frac{\e_0}2} \|f\|_{L^2}\times \vertiii{f}_{KN(\la,\e_0)},
\quad \text{if  supp }f\subset B(x_0,\delta).
\end{multline}
Here $B(x,\delta)$ denotes the $\delta$-ball about $x$ in our coordinates.  
We may assume that in our local coordinate system the line segment $(0,y)$, $|y|<4$ is a geodesic.

In order to prove \eqref{2.7} we also need to define a microlocal version of the above Kakeya-Nikodym norm.  
  We first
choose $0\le \beta\in C^\infty_0({\mathbb R}^2)$ satisfying
\begin{equation}\label{2.8}
\sum_{\nu\in {\mathbb Z}^2} \beta(z+\nu)=1, \quad \text{and } \,  \, \text{supp } \beta \subset \{x\in {\mathbb R}^2: \, |x|\le 2\}.
\end{equation}

To use this bump function, let $\Phi_t(x,\xi)=(x(t),\xi(t))$ denote the geodesic flow on the unit cotangent bundle.  Then if $(x,\xi)$ is a unit cotangent
vector with $x\in B(x_0,\delta)$ and $|\xi_1|<\delta$, with $\delta$ small enough, it follows that there is a unique $0<t<10$ so that $x(t)=(s,0)$ for some $s(x,\xi)$.  If then for this $t$,
$\xi(t)=(\xi_1(t),\xi_2(t))$, it follows that $\xi_2(t)$ is bounded from below.  Let us then set $\varphi(x,\xi)=(s(x,\xi), \xi_1(t)/|\xi(t)|)$.  Note that $\varphi$ then
is a smooth map from such unit cotangent vectors to ${\mathbb R}^2$.  Also, $\varphi$ is constant on the orbit of $\Phi$.  Therefore,
$|\varphi(x,\xi)-\varphi(y,\eta)|$ can be thought as measuring the distance from the geodesic in our coordinate patch through $(x,\xi)$ to that of the one
through $(y,\eta)$.

Let $\alpha(x)$ be a nonnegative $C^\infty_0$ function which is one in 
$B(x_0,\tfrac32\delta)$ and zero outside of $B(x_0,2\delta)$.  
Given $\theta=2^{-k}$ with $\la^{-\frac12}\le \theta \le 1$, and $\nu\in {\mathbb Z}^2$ let $\Upsilon\in C^\infty({\mathbb R})$ satisfy
\begin{equation}\label{2.9}
\Upsilon(s)=1, \quad s\in [c, c^{-1}], \quad \Upsilon(s)=0, \quad s\notin [c/2, 2c^{-1}],
\end{equation}
for some $c>0$ to be specified later.  We then put
\begin{equation}\label{2.10}
Q_\theta^\nu(x,\xi)=\alpha(x) \, \beta\bigl(\theta^{-1}\varphi(x,\xi)+\nu\bigr) \, \Upsilon(|\xi|/\la) .
\end{equation}
This is a function of unit cotangent vectors, and we also denote its homogeneous of degree zero extension to the cotangent bundle with the zero
section removed by $Q_\theta^\nu(x,\xi)$, $\xi\ne 0$, and the resulting pseudodifferential operator by $Q_\theta^\nu(x,D)$.
Then if $f$ is as in \eqref{2.7}, we define its microlocal Kakeya-Nikodym norm corresponding to frequency $\la$ and angle $\theta_0=\la^{-\frac12+\e_0}$ to be
\begin{multline}\label{2.11}
\vertiii{f}_{MKN(\la,\e_0)}
\\ 
=\sup_{\theta_0 \le \theta\le 1}\bigl(\, 
 \sup_{\nu \in {\mathbb Z}^2}\theta^{-\frac12} \|Q_{\theta}^\nu(x,D)f\|_{L^2({\mathbb R}^2)}\, \bigr) +\|f\|_{L^2({\mathbb R}^2)}, \quad \theta_0=\la^{-\frac12+\e_0}.
\end{multline}
Note that 
$$\sup_{\nu \in {\mathbb Z}^2}\theta^{-\frac12} \|Q_{\theta}^\nu(x,D)f\|_{L^2({\mathbb R}^2)}$$
 measures the maximal microlocal concentration of $f$ about all unit geodesics in the scale of $\theta$.  This is because
of the fact that if we consider the restriction of $Q^\nu_\theta$ to unit cotangent vectors and if $Q^\nu_\theta(x,\xi)\ne 0$, then $\supp Q^\nu_\theta$ is contained in an $O(\theta)$ tube in the space of unit cotangent vectors about the orbit $t\to \Phi_t(x,\xi)$.

Let us collect a few facts about these pseudodifferential operators.  First, the $Q^\nu_\theta$ belong to a bounded subset of $S^0_{1/2+\e_0,1/2-\e_0}$
(pseudodifferential operators of order zero and type $(1/2+\e_0,1/2-\e_0)$),
if $\la^{-\frac12 +\e_0}\le \theta\le 1$, with $\e_0>0$ fixed.  Therefore, there is a uniform constant $C_{\e_0}$ so that
\begin{equation}\label{2.12}
\|Q^\nu_\theta(x,D)g\|_{L^2}\le C_{\e_0}\|g\|_{L^2}, \quad \la^{-\frac12 +\e_0}\le \theta\le 1.
\end{equation}
Similarly, if $P^\nu_\theta=(Q^\nu_\theta)^*\circ Q^\nu_\theta$, then by \eqref{2.8}, for such $\theta$,  $\sum_\nu P^\nu_\theta$ belongs to a
bounded subset of $S^0_{1/2+\e_0,1/2-\e_0}$, and so we also have the uniform bounds
\begin{equation}\label{2.13}
\bigl\|\sum_{\nu\in {\mathbb Z}^2}P^\nu_\theta(x,D)g \bigr\|_{L^2}\le C_{\e_0}\|g\|_{L^2}, \quad \la^{-\frac12 +\e_0}\le \theta\le 1.
\end{equation}

We can relate the microlocal Kakeya-Nikodym norm to the Kakeya-Nikodym norm if we realize that if the $\delta>0$ above is small enough then
there is a unit length geodesic $\gamma_\nu$ so that $Q^\nu_\theta(x,\xi)=0$ for $x\notin {\mathcal T}_{C\theta_\nu}(\gamma)$, with $C$ being
a uniform constant.  As a result, since $Q^\nu_\theta(x,\xi)=0$ if $|\xi|$ is not comparable to $\la$,
we can improve \eqref{2.12} and deduce that for every $N=1,2,\dots$, % it is not difficult to check
that there is a uniform constant $C'$ so that we have
\begin{equation}\label{2.14}
\|Q^\nu_\theta(x,D)g\|_{L^2} \le C_{\e_0}\bigl(\int_{{\mathcal{T}_{C'\theta}(\gamma_\nu)}} |g|^2 dy\bigr)^{\frac12}
+C_N\la^{-N}\|g\|_{L^2}, \quad  \la^{-\frac12 +\e_0}\le \theta\le 1,
\end{equation}
since the kernel $K^\nu_\theta(x,y)$ of $Q^\nu_\theta(x,D)$ is $O(\la^{-N})$ for any $N$ if $y$ is not in
${\mathcal T}_{C'\theta}(\gamma_\nu)$, with $C'$ sufficiently large but fixed.  (See Figure 1.)
Since $$\theta^{-\frac12} \bigl(\int_{{\mathcal{T}_{C'\theta}(\gamma_\nu)}} |g|^2 dy\bigr)^{\frac12}\lesssim
\sup_{\gamma\in \varPi}\bigl(\theta_0^{-1}\int_{{\mathcal{T}_{\theta_0}(\gamma)}} |g|^2 dy\bigr)^{\frac12}, \quad
\la^{-\frac12+\e_0}=\theta_0\le \theta\le 1,$$
%Whence, 
we have
\begin{equation}\label{2.15}
 \sup_{\nu \in {\mathbb Z}^2}\theta^{-\frac12} \|Q_{\theta}^\nu(x,D)f\|_{L^2({\mathbb R}^2)}
\le C_{\e_0} \vertiii{g}_{KN(\la,\e_0)}, \quad \la^{-\frac12+\e_0}\le \theta \le 1,
\end{equation}
%\begin{equation}\label{2.15}
%\sum_{1\le 2^k\le \la^{\frac12 -\e_0}} \sup_{\nu \in {\mathbb Z}^2}\bigl(2^{\frac{k}2} \|Q_{2^{-k}}^\nu(x,D)f\|_{L^2({\mathbb R}^2)}\bigr)
%\le C_{\e_0} \vertiii{g}_{KN(\la,\e_0)},
%\end{equation}
meaning that we can  dominate the microlocal Kakeya-Nikodym norm by the Kakeya-Nikdodym norm.

From this, we conclude that we would have \eqref{2.7} if we could show
\begin{multline}\label{main}
\Bigl\|\int \la^{\frac12}e^{i\la \psi(x,y)}a(x,y) f(y)\, dy\Bigr\|_{L^4(B(0,\delta))}^2
\\
\lesssim_{\e_0} \la^{\frac{\e_0}2} \|f\|_{L^2}\times \vertiii{f}_{MKN(\la,\e_0)},
\quad \text{if  supp }f\subset B(x_0,\delta).
\end{multline}
We note also that since $\chi_\la e_\la =e_\la$, this inequality of course yields
the following microlocal strengthening of Theorem~\ref{theorem1.1}:

\begin{theorem}\label{microthm}
For every $0<\e_0\le \tfrac12$ we have
\begin{equation}\label{m.1}
\|e_\la\|_{L^4(M)}\lesssim_{\e_0} \la^{\frac{\e_0}4}
\|e_\la\|_{L^2(M)}^{\frac12}
\times
\vertiii{e_\la}_{MKN(\la,\e_0)}^{\frac12}.
\end{equation}
if $\vertiii{e_\la}_{MKN(\la,\e_0)}$ is as in \eqref{2.11}.
\end{theorem}

%\begin{figure}[h]
%\resizebox{3cm}{!}{\input{singletube.pdf_t}}
%\caption{${\mathcal T}_{C'\theta(\gamma_\nu)}$}
%\label{fig1}
%\end{figure}

\begin{figure}
\begin{center}
\resizebox{3.5cm}{!}{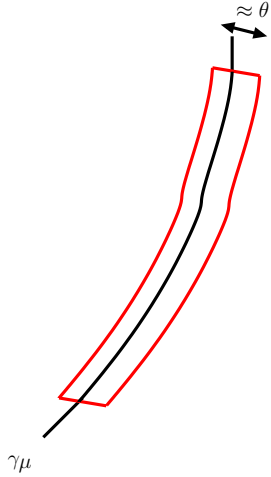}
\end{center}\caption{${\mathcal T}_{C'\theta(\gamma_\nu)}$}
\label{fig1}
\end{figure}

\newsection{Proof of the refined  two-dimensional microlocal Kakeya-Nikodym estimates}

Let us now prove the estimates in \eqref{main}.  We shall follow arguments from \S 6 of \cite{MSS}.

We first note that if as in \eqref{2.7}, $\text{supp }f\subset B(x_0,\delta)$, 
and if 
\begin{equation}\label{3.1}\theta_0=\la^{-\frac12+\e_0}\end{equation}
with $\e_0>0$ fixed
$$\chi_\la f=\sum_{\nu \in {\mathbb Z}^2}\chi_\la \bigl( Q^\nu_{\theta_0}(x,D)f\bigr)+R_\la f,$$
where, if $c>0$ in \eqref{2.9} is small enough,   and $N=1,2,3,\dots$
$$\|R_\la f\|_{L^\infty}\lesssim \la^{-N}\|f\|_{L^2}.$$
Therefore, in order to prove \eqref{2.7}, it suffices to show that 
\begin{equation}\label{3.2}
\Bigl\|\sum_{\nu,\nu'\in {\mathbb Z}^2}\chi_\la Q^\nu_{\theta_0}f \, \chi_\la Q^{\nu'}_{\theta_0}f\Bigr\|_{L^2}\lesssim_{\e_0} \la^{\frac{\e_0}2} \|f\|_{L^2}\times \vertiii{f}_{MKN(\la,\e_0)}.
\end{equation}

We shall split the sum in the left based on the size of $|\nu-\nu'|$.  Indeed the left side of \eqref{3.2} is dominated by
\begin{equation}\label{3.3}
\Bigl\| \sum_\nu \bigl(\chi_\la Q^\nu_{\theta_0}f\big)^2\Bigr\|_{L^2}
+\sum_{\ell=1}^\infty \Bigl\| \sum_{|\nu-\nu'|\in [2^\ell, 2^{\ell+1})}\chi_\la Q^\nu_{\theta_0}f \, \chi_\la Q^{\nu'}_{\theta_0}f\Bigr\|_{L^2}.
\end{equation}

The square of the first term in \eqref{3.3} is
$$\sum_{\nu,\nu'} \int \bigl(\chi_\la Q^\nu_{\theta_0}f\bigr)^2 \, \overline{\bigl(\chi_\la Q^{\nu'}_{\theta_0}f\bigr)^2} \, dx.$$
Next we need an orthogonality result which is similar to Lemma 6.7 in \cite{MSS}, which says that
%The proof of Lemma 6.7 in \cite{MSS} shows that (since $\e_0>0$) 
if $A$ is large enough we have
\begin{equation}\label{3.4}\sum_{|\nu-\nu'|\ge A}\left|\int \bigl(\chi_\la Q^\nu_{\theta_0}f\bigr)^2  \, \overline{\bigl(\chi_\la Q^{\nu'}_{\theta_0}f\bigr)^2} \, dx\right|
\lesssim_{\e_0,N} \la^{-N}\|f\|^4_{L^2}.\end{equation}
We shall postpone the proof of this result until the end of the section when we will have recorded the
information about the kernels of $\chi_\la Q^\nu_{\theta}$ that will be needed for the proof.

%Therefore, 
Since by \cite{SEig}
$$\|\chi_\la\|_{L^2\to L^4}=O(\la^{\frac18}),$$
if we use \eqref{3.4}
we conclude that
the first term in \eqref{3.3} is majorized  by \eqref{2.13} and \eqref{2.15} by
\begin{multline}\label{3.5}
\la^{\frac12}\sum_\nu \|Q^\nu_{\theta_0}f\|_{L^2}^2\, \|Q^\nu_{\theta_0}f\|_{L^2}^2 +\la^{-N}\|f\|_{L^2}^4
\lesssim \la^{\frac12}\|f\|_{L^2}^2 \times\sup_{\nu\in {\mathbb Z}^2}\|Q^\nu_{\theta_0}f\|_{L^2}^2+\la^{-N}\|f\|_{L^2}^4
\\
=\la^{\e_0}\|f\|_{L^2}^2 \times \la^{\frac12-\e_0}\sup_{\nu\in {\mathbb Z}^2}\|Q^\nu_{\theta_0}f\|_{L^2}^2+\la^{-N}\|f\|_{L^2}^4
\end{multline}
Therefore,  the first term in \eqref{3.3} satisfied the desired bounds.

Using \eqref{2.15} again, the proof of \eqref{main} and hence \eqref{2.7} would be complete if we could estimate the other terms in \eqref{2.11} and show that for
\begin{multline}\label{3.6}
\Bigl\| \sum_{|\nu-\nu'|\in [2^\ell, 2^{\ell+1})}\chi_\la Q^\nu_{\theta_0}f \, \chi_\la Q^{\nu'}_{\theta_0}f\Bigr\|_{L^2}^2
\\
\lesssim_{\e_0} 
%\la^{\e_0} 
\|f\|_{L^2}^2 \times (2^\ell\theta_0)^{-1}\sup_{\nu\in {\mathbb Z}^2}\|Q^\nu_{2^\ell\theta_0}f\|_{L^2}^2+\la^{-N}\|f\|_{L^2}^4.
\end{multline}
Note that if $2^\ell\theta_0\gg 1$ the left side of \eqref{3.6} vanishes and thus, as in \eqref{2.15}, we are just considering $\ell\in {\mathbb N}$
satisfying $1\le 2^\ell\le \la^{\frac12-\e_0}$.  In proving this, we may assume that $\ell$ is larger than a fixed constant, since the bound for small
$\ell$ (with an extra factor of $\la^{\e_0}$ in the right) follows from what we just did.  We can handle the sum over $\ell$ in \eqref{3.3} due to the fact that
the right side of \eqref{3.6} does not include a factor $\la^{\e_0}$.

We now turn to estimating the non-diagonal terms in \eqref{3.3}.
We first note that 
%if $Q^{\nu}_\theta$ is defined as in \eqref{3.3} with 
%$\theta_0$ replaced by $\theta$, then 
by \eqref{2.8}
%$$Q^{\nu}_{\theta_0}=\sum_{\mu\in \Zt} Q^{\mu}_\theta\circ Q^\nu_{\theta_0}.$$
\begin{equation*}
\chi_\la Q^\nu_{\theta_0}f=\sum_{\mu\in \Zt}\chi_\la Q^\mu_\theta Q^\nu_{\theta_0}f +O_N(\la^{-N}\|f\|_2),
\quad \text{if } \, \, \, \text{supp } f\subset B(x_0,\delta).
\end{equation*}

Furthermore,  if, as we may, we assume that $\ell\in {\mathbb N}$ is sufficiently large, then given $N_0\in {\mathbb N}$, there are fixed constants $c_0>0$ and $N_1<\infty$ (with $c_0$ depending only on $N_0$ and the cutoff $\beta$ in the definition
of these pseudodifferential operators) so that if
$$\theta_\ell = \theta_02^\ell,$$
then
\begin{multline}\label{3.7}
\sum_{|\nu-\nu'|\in [2^\ell, 2^{\ell+1})} \chi_\la Q_{\theta_0}^{\nu}f \, \,  \chi_\la Q^{\nu'}_{\theta_0}f
\\
=
\sum_{\{\mu,\mu'\in \Zt: \, N_0\le |\mu-\mu'|\le N_1\}}\sum_{|\nu-\nu'|\in [2^\ell, 2^{\ell+1})} 
\chi_\la Q^{\mu}_{c_0\theta_\ell}Q^\nu_{\theta_0}f \, \, \chi_\la Q^{\mu'}_{c_0\theta_\ell}Q^{\nu'}_{\theta_0}f
+O_N(\la^{-N}\|f\|_{L^2}^2),
\end{multline}
for each $N\in {\mathbb N}$.
Also, given $\mu\in \Zt$, there is a $\nu_0(\mu)\in \Zt$ so that
$$\|Q^{\mu}_{c_0\theta_\ell}Q^\nu_{\theta_0}f\|_{L^2}\le C_N\la^{-N}\|f\|_{L^2}, 
\quad \text{if } \, |\nu-\nu_0(\mu)|\ge C2^\ell,$$
for some uniform constant $C$.  If $|\mu-\mu'|\le N_1$, then $|\nu_0(\mu)-\nu_0(\mu')|\le C2^\ell$ for some uniform constant $C$.
Since $\|(Q^{\nu'}_\theta)^*\circ Q^\nu_\theta\|_{L^2\to L^2}=O(\la^{-N})$ for every $N$ if $|\nu-\nu'|$  is larger than a fixed constant, it follows that
\begin{multline}\label{3.8}
\iint \Bigl| \,  \sum_{|\nu_0(\mu)-\nu|, \, |\nu_0(\mu')-\nu'|\le C2^\ell} \, \sum_{|\nu-\nu'|\in [2^\ell, 2^{\ell+1})}Q^\nu_{\theta_0}f(x)Q^{\nu'}_{\theta_0}f(y)\Bigr|^2 \, dx dy
\\
\lesssim \sum_{|\nu-\nu_0(\mu)|, \, |\nu'-\nu_0(\mu)|\le C'2^\ell} \|Q^\nu_{\theta_0}f\|_{L^2}^2\, \|Q^{\nu'}_{\theta_0}f\|_{L^2}^2\, + \, O_N(\la^{-N}\|f\|_{L^2}^2),
\quad \text{if } \, |\mu-\mu'|\le C_0,
\end{multline}
for every $N$ if $C'$ is a sufficiently large but fixed constant.  Also, using  \eqref{2.13}, we deduce that
$$\sum_{\mu\in \Zt}\sum_{|\nu_0(\mu)-\nu|\le C'2^\ell}\|Q^\nu_{\theta_0}f\|^2_{L^2}\lesssim \|f\|_{L^2}^2.
$$
We clearly also have 
$$\sum_{|\nu(\mu)-\nu'|\le C'2^\ell}\|Q^{\nu'}_{\theta_0}f\|^2_{L^2}\lesssim \sup_{\mu\in \Zt} \|Q^\mu_{2^\ell \theta}f\|_{L^2}^2.$$
 Using these two inequalities and \eqref{3.8}, we deduce that
\begin{multline}\label{3.9}
\sum_{|\mu-\mu'|\le N_1}
\Bigl\| \,  \sum_{|\nu_0(\mu)-\nu|, \, |\nu_0(\mu')-\nu'|<C2^\ell} \, \sum_{|\nu-\nu'|\in [2^\ell, 2^{\ell+1})}Q^\nu_{\theta_0}f(x)Q^{\nu'}_{\theta_0}f(y)\Bigr\|_{L^2(dxdy)}
\\
\lesssim \|f\|_{L^2}\times \sup_{\mu\in \Zt} \|Q^\mu_{2^\ell \theta}f\|_{L^2} \, + \, O_N(\la^{-N}\|f\|^2_{L^2}).
\end{multline}

%Next, we note that it follows from the proof of Lemma 6.7 in \cite{MSS} that 
%we have the following analog of \eqref{3.4} whenever $\theta$
In addition to  \eqref{3.4} we shall need another orthogonality result whose proof we postpone until
the end of the section, which says that whenever $\theta$
is larger than a fixed positive multiple of $\theta_0$ in \eqref{3.1} and $N_1$ is fixed
\begin{multline}\label{3.10}
\Bigl|\int \bigl(\chi_\la Q^{\mu}_\theta g_1\, \chi_\la Q^{\mu'}_\theta g_2\bigr) \, \overline{\bigl(\chi_\la Q^{\tilde \mu}_\theta g_3 \,
\chi_\la Q^{\tilde \mu'}_\theta g_4\bigr)} \, dx\Bigr|
\\
\lesssim_{N}\la^{-N}\prod_{j=1}^4\|g_j\|_{L^2}, \quad \text{if  }
|\mu-\tilde \mu|+|\mu'-\tilde \mu'|\ge C, \, \, \text{and } \, \, |\mu-\mu'|, \, |\tilde \mu -\tilde \mu'|\le N_1,
\end{multline}
for every $N=1,2,\dots$,
with $C$ being a sufficiently large uniform constant (depending on $N_1$ of course).

Using \eqref{3.9} and \eqref{3.10}, we conclude that we would have \eqref{3.6} (and consequently \eqref{2.7}) if we could prove the following

\begin{proposition}\label{prop3.1}
Let 
\begin{equation}\label{3.11}
\Bigl(T^{\mu,\mu'}_{\la,\theta} F\Bigr)(x)=\iint \bigl(\chi_\la Q^{\mu}_\theta\bigr)(x,y) \, \bigl(\chi_\la Q^{\mu'}_\theta \bigr)(x,y') \, F(y,y')\, dydy',
\end{equation}
where
$$\bigl(\chi_\la Q^{\mu}_\theta\bigr)(x,y)$$
denotes the kernel of $\chi_\la Q^{\mu}_\theta$.
%, with $Q^{\mu}_\theta$ defined by the analog of \eqref{3.3} with $\theta_0$ replaced by $\theta$.
Then if $\delta>0$ is sufficiently small and if $\theta$ is larger than a fixed positive constant times $\theta_0$ in \eqref{3.1} and if $N_0\in {\mathbb N}$ is sufficiently large and if
$N_1>N_0$ is fixed, we have
\begin{multline}\label{3.12}
\bigl\| T^{\mu,\mu'}_{\la,\theta} F\bigr\|_{L^2(B(0,\delta))} 
\lesssim_{\e_0}
\theta^{-\frac12}\|F\|_{L^2},
\quad
\text{if  } N_0\le |\mu-\mu'|\le N_1, 
\\ \text{and } \, \, F(y,y')=0, \, \text{if } \, \, (y,y')\notin B(x_0,2\delta)\times B(x_0,2\delta).
\end{multline}
\end{proposition}

To prove this we shall need some information about the kernel of $\chi_\la Q^\mu_\theta$.  One thing will be that,
by \eqref{2.10}, the kernel is highly concentrated near the geodesic in $M$
\begin{equation}\label{geod}
\gamma_\mu =\{x_\mu(t): \, -2\le t\le 2, \Phi_t(x_\mu,\xi_\mu)=(x_\mu(t),\xi_\mu(t)), \, \, \theta^{-1}\varphi(x_\mu,\xi_\mu)+\mu=0\, \},
\end{equation}
which corresponds to $Q^\mu_\theta$.  We also will exploit the oscillatory behavior of the kernel near $\gamma_\mu$.  

Specifically, we require the following

\begin{lemma}\label{kerlemma}
Let $\theta \in [C_0\la^{-\frac12+\e_0}, \frac12]$, where $C_0$ is a sufficiently large fixed constant, and, as above,
$\e_0>0$.  Then there is a uniform constant $C$ so that for each $N=1,2,3,\dots$ we have
\begin{equation}\label{supp}
|(\chi_\la Q^\mu_\theta)(x,y)|\le C_N\la^{-N}, \quad\text{if  }\, 
x\notin {\mathcal T}_{C\theta}(\gamma_\mu),
\, \, \text{or } \, y\notin {\mathcal T}_{C\theta}(\gamma_\mu).
\end{equation}
Furthermore, 
\begin{equation}\label{osc}
\bigl(\chi_\la Q^\mu_\theta\bigr)(x,y)=\la^{\frac12}e^{i\la d_g(x,y)}a_{\mu,\theta}(x,y)+O_N(\la^{-N}),
\end{equation}
where one has the uniform bounds
\begin{equation}\label{osc1}
|\nabla^\alpha_y a_{\mu,\theta}(x,y)|\le C_\alpha \theta^{-|\alpha|},
\end{equation}
and
\begin{equation}\label{osc2}
|\partial_t^ja_{\mu,\theta}(x,x_\mu(t))|\le C_j, \, \, x\in \gamma_\mu,
\end{equation}
if, as in \eqref{geod}, $\{x_\mu(t)\}=\gamma_\mu$.
\end{lemma}

\begin{proof}
To prove the lemma it is convenient to choose Fermi normal coordinates so that the
geodesic becomes the segment $\{(0,s): \, |s|\le 2\}$.  
Let us also write $\theta$ as
$$\theta =\la^{-\frac12+\delta},$$
where, because of our assumptions $c_1\le \delta\le 1/2$ for an
appropriate $c_1>0$.
Then, in these coordinates
$Q^\mu_\theta(x,D)$ has symbol satisfying
\begin{equation}\label{e1}
q^\mu_\theta(x,\xi)=0, \quad \text{if } \, \, 
\bigl|\xi_1/|\xi|\, \bigr| \ge C\la^{-\frac12+\delta},
\, \,  \, |x_1|\ge C\la^{-\frac12+\delta} \, \, \,
\text{or } \, \, |\xi|/\la \notin [C^{-1},C],
\end{equation}
for some uniform constant $C$, and, additionally,
\begin{equation}\label{e2}
|\partial^{j}_{x_1}\partial_{x_2}^k\partial^l_{\xi_1}\partial^m_{\xi_2} q^\mu_\theta(x,\xi)|\le C_{ j, k, l, m}
(1+|\xi|)^{j(\frac12-\delta)-l(\frac12+\delta)-m}.
\end{equation}

Next we recall that $\chi_\la =\rho(\la-\sqrt{-\Delta_g})$ where $\rho\in {\mathcal S}(\R)$
satisfies $\hat \rho \subset (1,2)$ and that the injectivity radius of $(M,g)$ is
ten or more.  Therefore, we can use Fourier integral parametrices for the
wave equation to see that
the kernel of $\chi_\la$ is of the form
$$\chi_\la(x,y)=\iint e^{iS(t,x,\xi)-iy\cdot \xi +it\la}\Hat\rho(t)\alpha(t,x,y,\xi)
\, d\xi dt,
$$
where $\alpha \in S^1_{1,0}$ and $S$ is homogeneous of degree one in $\xi$ 
is a generating function for the canonical relation for the half wave group $e^{-it\sqrt{-\Delta_g}}$.
Thus,
\begin{equation}\label{e3}
\partial_tS(t,x,\xi)=-p(x,\nabla_x S(t,x,\xi)), \quad S(0,x,\xi)=x\cdot \xi,
\end{equation}
and
\begin{equation}\label{e4}
\Phi_t(x,\nabla_xS)=(\nabla_\xi S,\xi), %\quad \text{if }  \, t=d_g(x,y),
\end{equation}
with, as before, $\Phi_t$ denoting geodesic flow on the cotangent bundle.
Furthermore,
\begin{equation}\label{e5}
\text{det} \, \frac{\partial S}{\partial x\partial \xi} \ne 0.
\end{equation}

By \eqref{e1}-\eqref{e2} and the proof of the Kohn-Nirenberg theorem,
we have that 
\begin{align}\label{e6}
\bigl(\chi_\la Q^\mu_\theta\bigr)(x,y)
&=\iint e^{iS(t,x,\xi)-iy\cdot \xi +i\la t}\Hat \rho(t)
q(t,x,y,\xi) \, d\xi dt +O(\la^{-N}),
\\
&=\la^2\iint e^{i\la(S(t,x,\xi)-y\cdot \xi + t)}\Hat \rho(t)
q(t,x,y,\la\xi) \, d\xi dt +O(\la^{-N}), \notag
\end{align}
where for all $t$ in the support of $\Hat \rho$,
\begin{equation}\label{e7}
q(t,x,y,\xi)=0 \, \, \,
\text{if } \, \, |\xi_1/|\xi|\, |\ge C\la^{-\frac12+\delta},
\, \, \, |x_1|\ge C\la^{-\frac12+\delta}, \, \, \,
\text{or } \, |\xi|/\la \notin [C^{-1},C],
\end{equation}
with $C$ as in \eqref{e2}, and, also
\begin{equation}\label{e8}
|\partial^{j}_{x_1}\partial_{x_2}^k\partial^l_{\xi_1}\partial^m_{\xi_2} q(t,x,y,\xi)|\le C_{ j, k, l, m}
(1+|\xi|)^{j(\frac12-\delta)-l(\frac12+\delta)-m}.
\end{equation}

Let us now prove \eqref{supp}.  We have the assertion if 
$y\notin {\mathcal T}_{C\la^{-\frac12+\delta}}(\gamma_\mu)$ by \eqref{e7}.  To prove
that remaining part of \eqref{e7} which
says that this is also the case when $x$ is not in such a tube, we note that
by \eqref{e4}, if $d_g(x_0,y_0)=t_0$ and $x_0,y_0\in \gamma_\mu$, then
$$\nabla_\xi\bigl(S(t_0,x_0,\xi)-y_0\cdot \xi\bigr)=0 \quad
\text{if } \, \xi_1=0. $$
By \eqref{e5} we then have
$$|\nabla_\xi(S(t_0,x,\xi)-y_0\cdot \xi)|\approx d_g(x,x_0)
\quad \text{if } \, \xi_1=0. $$
We deduce from this that if $|\xi_1|/|\xi|\le C\la^{-\frac12+\delta}$, $|y_1|\le C\la^{-\frac12+\delta}$ and $|\xi|\in [C^{-1},C]$, then there is
a $c_0>0$ and a $C_0<\infty$ so that
$$|\nabla_\xi(S(t_0,x,\xi)-y\cdot \xi)|\ge c_0\la^{-\frac12+\delta}
\quad \text{if } \, \, x\notin {\mathcal T}_{C_0\la^{-\frac12+\delta}}(\gamma_\mu). $$
From this we obtain the remaining part of \eqref{supp} via a simple
integration by parts argument if we use the support properties \eqref{e7}
and size estimates \eqref{e8} of $q(t,x,y,\xi)$.  We note that every
time we integrate by parts in $\xi$ we gain by $\la^{-2\delta}$
which implies \eqref{supp} since $q$ vanishes unless $|\xi|\approx \la$
and $\delta$ is bounded below by a fixed positive constant.

To finish the proof of the lemma and obtain \eqref{osc}-\eqref{osc2},
we note that if we let
$$\Psi(t,x,y,\xi)=S(t,x,\xi)-y\cdot \xi +t$$
denote the phase function of the second oscillatory integral in \eqref{e6},
then at a stationary point where
$$\nabla_{\xi,t}\Psi=0, $$
we must have $\Psi=d_g(x,y)$, due to the fact that $S(t,x,\xi)-y\cdot \xi=0$ and $t=d_g(x,y)$ at points where the $\xi$-gradient vanishes.
Additionally, 
it is not difficult to check that the mixed Hessian of the phase satisfies
$$\text{det }\Bigl(\frac{\partial^2\Psi}{\partial(\xi,t)\partial(\xi,t)}\Bigr)
\ne 0$$
on the support of the integrand.
This follows from the proof of \cite[Lemma 5.1.3]{SBook}.  Moreover, since, 
modulo $O(\la^{-N})$ error terms $(\chi_\la Q^\mu_\theta)(x,y)$ equals
\begin{equation}\label{e9}\la^2\iint e^{i\la \Psi}\Hat \rho(t) \, q(t,x,y,\la\xi) \, d\xi dt,
\end{equation}
we obtain \eqref{osc}-\eqref{osc1} by the proof of this result if we
use stationary phase and \eqref{e7}-\eqref{e8}.  Indeed, by \eqref{e4},
\eqref{e9} has a stationary phase expansion (see 
\cite[Theorem 7.7.5]{H1})
where the leading term is 
a fixed constant times
\begin{equation}\label{e10}\la^{\frac12} e^{i\la t}q(t,x,y,\la\xi), \quad 
\text{if } \, t=d_g(x,y) \, \, \text{and } \,
\Phi_{-t}(y,\xi)=(x,\nabla_xS(t,x,\xi)).
\end{equation}
From this, we see that the leading term in the asymptotic expansion
must satisfy \eqref{osc1}, and subsequent terms in the expansion
will satisfy better estimates where the right hand side involves increasing
negative powers of $\lambda^{2\delta}$ (by \cite[(7.7.1)]{H1} and \eqref{e8}), from which we deduce that \eqref{osc1}
must be valid.  Since $\xi_1=0$ and $p(y,\xi)=1$ (by \eqref{e4})
in \eqref{e10} when $x,y\in \gamma_\mu$, we similarly deduce
from \eqref{e8} that the leading term in the stationary phase
expansion must satisfy \eqref{osc2}, and since the other terms satisfy
better bounds involving increasing powers of $\la^{-2\delta}$, we
similarly obtain \eqref{osc2}, which completes the proof of the lemma.
\end{proof}

Let us now collect some simple consequences of Lemma \ref{kerlemma}.  First, in addition to \eqref{supp}, the kernel $(\chi_\la Q^\mu_\theta)(x,y)$
is also $O(\la^{-N})$ unless the distance between $x$ and $y$ is comparable to one
by \eqref{2.6}.  From this we deduce  that if $N_0\in {\mathbb N}$ is sufficiently large
\begin{multline}\label{3.13}\bigl(\chi_\la Q^{\mu}_\theta\bigr)(x,y) \, \bigl(\chi_\la Q^{\mu'}_\theta\bigr)(x,y')=O(\la^{-N}),
\\
\text{unless } \, \, \text{Angle}(x;y,y') \in [\theta, C_2\theta], \, \, \text{and } \, \, x,y,y'\in {\mathcal T}_{C_2\theta}(\gamma_\mu),
\, \, \, \text{if } \, |\mu-\mu'|\in [N_0,N_1],
\end{multline}
if $\text{Angle}(x,y,y')$ denotes the angle at $x$ of the geodesic connecting $x$ and $y$ and the one connecting $x$ and $y'$, and where $C_2=C_2(N_1)$.

This is because in this case, if $x\in {\cal T}_{C\theta}(\gamma_\mu)\cap {\cal T}_{C\theta}(\gamma_{\mu'})$
then the tubes must be disjoint at a distance bounded below by a fixed positive multiple of  $\theta$ if $N_0$ is large enough, and in this region their separation is bounded by a fixed
constant times $\theta$ if $N_1$ is fixed.  See the figure below.

%\begin{figure}[h]
%\resizebox{7cm}{!}{\input{intersectingtubes.pdf_t}}
%\caption{$\theta$-tubes intersecting at angle $\ge N_0 \theta$}
%\label{fig2}
%\end{figure}

\begin{figure}
\begin{center}
\resizebox{6cm}{!}{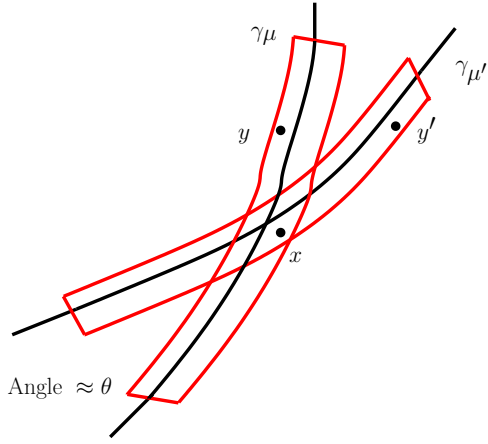}
\end{center}
\caption{$\theta$-tubes intersecting at angle $\ge N_0 \theta$}
\label{fig2}
\end{figure}

To exploit this key fact, as above, let us choose Fermi normal coordinates about $\gamma_{\mu}$ so that the geodesic becomes the segment
$\{(0,s): \, |s|\le 2\}$.  Then, as in \eqref{2.5}, let 
$$\psi(x; y)=d_g\bigl((x_1,x_2),(y_1,y_2)\bigr)$$
be the Riemannian distance function  written in these coordinates.  Then if $x,y,y'$ are close to this segment and if the distance between $x$ and $y$ and $x$
and $y'$ are both comparable to one and if, as well, $y$ is close to $y'$, it follows from Gauss' lemma that
\begin{equation}\label{3.14}\text{Angle}\bigl(x;(y_1,y_2),(y'_1,y'_2)\bigr)
\approx \bigl| \frac{\partial}{\partial y_1}\frac{\partial}{\partial x_2}\psi(x,y) \, - \,  \frac{\partial}{\partial y_1}\frac{\partial}{\partial x_2}\psi(x,y')\bigr|.
\end{equation}
As a result, by \eqref{3.13}, we have that there must be a constant $c_0>0$ so that
\begin{multline}\label{3.15}\bigl(\chi_\la Q^{\mu}_\theta\bigr)(x,y) \, \bigl(\chi_\la Q^{\mu'}_\theta\bigr)(x,y')=O(\la^{-N}),
\\
\text{if } \, \, \Bigl| \frac{\partial}{\partial y_1}\frac{\partial}{\partial x_2}\psi(x,y) \, - \,  \frac{\partial}{\partial y_1}\frac{\partial}{\partial x_2}\psi(x,y')\Bigr|
\le c_0 \theta, \, \, \text{and } \, \, |\mu-\mu'|\in [N_0,N_1],
\end{multline}
with, as above, $N_0\in {\mathbb N}$ sufficiently large and $N_1$ fixed.
Another consequence of Gauss' lemma is that if $x$ and $y$, as in \eqref{3.14} are close to this segment and a distance which is comparable to one
from each other, 
then
\begin{equation}\label{3.16}
\frac\partial{\partial x_1}\frac\partial{\partial y_1}\psi(x,y)\ne 0.
\end{equation}
We shall also need to make use of the fact that, in these Fermi normal coordinates, we also have
 \begin{multline}\label{3.17}
 \frac\partial{\partial x_2}\frac\partial{\partial y_1}\psi((0,x_2),(0,y_2))=
 \frac{\partial}{\partial x_1}\psi((0,x_2), (0,y_2))=0, 
 \\ %\text{and } \, \frac\partial{\partial x_1}\frac\partial{\partial y_1}\psi((0,x_2),(0,y_2))=2,
 \, \, \text{if } \, \, d_g((0,x_2),(0,y_2))\approx 1.
 \end{multline}

Next, by \eqref{osc}-\eqref{osc2},
modulo terms which are $O(\la^{-N})$ we can write
$$\bigl(\chi_\la Q^\mu_\theta\bigr)(x,y)\, \bigl(\chi_\la Q^{\mu'}_\theta\bigr)(x,y')= \la \, e^{i\la(\psi(x,y)+\psi(x,y'))}b_\mu(x;y,y'),$$
where, by \eqref{3.13} and \eqref{3.15},
\begin{multline}\label{3.18}
b_\mu(x;y,y')=0, \text{if } \, d_g(x,y) \, \, \text{or } \, d_g(x,y')\notin [1,2], \quad \text{or } \, \, |x_1|+|y_1|+|y'_1|\ge c_0^{-1}\theta,
\\
\text{or } \, \, \Bigl| \frac{\partial}{\partial y_1}\frac{\partial}{\partial x_2}\psi(x,y) \, - \,  \frac{\partial}{\partial y_1}\frac{\partial}{\partial x_2}\psi(x,y')\Bigr|
\le c_0 \theta ,
\end{multline}
and, since we are working in Fermi normal coordinates
\begin{equation}\label{3.19}
\Bigl| \, \frac{\partial^j}{\partial x_1^j} \frac{\partial^k}{\partial x_2^k} \, b_\mu(x,y,y')\, \Bigr|\le C_0 \theta^{-j}, \quad 0\le j,k\le 3.
\end{equation}
The constants $C_0$ and $c_0$ can be chosen to be independent of $\mu\in \Zt$ and $\theta\ge \la^{-\frac12+\e_0}$ if $\e_0>0$.  But then, by \eqref{3.18} and \eqref{3.19}
if $y_2$ and $y_2'$ are fixed and close to one another, and if we set
\begin{equation*}
\Psi(x;s,t)=\psi\bigl(x,(s+t,y_2)\bigr)+\psi\bigl(x,(s-t,y_2')\bigr), \quad \text{and } \, b(x;s,t)=b_\mu(x;s+t,y_2,s-t,y_2'),
\end{equation*}
then we have that there is a fixed constant $C$ so that
\begin{multline}\label{3.20}
b(x;s,t)=0, \quad \text{if  }\, \, |x_1|+|s|+|t|\ge C\theta,
\\
\text{and } \, \, 
\Bigl| \, \frac{\partial^j}{\partial x_1^j} \frac{\partial^k}{\partial x_2^k} \, b(x;s,t)\, \Bigr|\le C \theta^{-j}, \quad 0\le j,k\le 3,
\end{multline}
while, by \eqref{3.16} and \eqref{3.17}
\begin{multline}\label{3.21}
\frac\partial{\partial x_2}\frac\partial{\partial s}\Psi(0,x_2;0,0)=
\frac\partial{\partial x_2}\frac\partial{\partial t}\Psi(0,x_2;0,0)=
\frac\partial{\partial x_1}\Psi(0,x_2;0,0)=0, 
\\
\text{but } \, \, 
\frac\partial{\partial x_1}\frac\partial{\partial s}\Psi(0,x_2;0,0)\ne 0, \quad \text{if } \, \, b(0,x_2;0,0)\ne 0,
\end{multline}
and, moreover, by \eqref{3.18},
\begin{equation}\label{3.22}
|\frac\partial{\partial x_2}\frac\partial{\partial t}\Psi(x; s,t)|\ge c\theta,  
\quad \text{if } \, \, b(x;s,t)\ne 0.
\end{equation}
Also, if, as we may, because of the support assumption in \eqref{3.12}, we assume that $|y_2-y_2'|\le  \delta$ then
\begin{equation}\label{3.23}
\Bigl|\frac\partial{\partial x_1}\frac\partial{\partial t}\Psi(x;s,0)\Bigl|\le C\delta, \quad \text{if } \, b(x;s,t)\ne 0,
\end{equation}
since the quantity in the left vanishes identically when $y_2=y_2'$.

Another consequence of Gauss' lemma is that if $y,y',x$ are close to the 2nd coordinate axis and if the distance between both $x$ and both $y$ and $y'$ are comparable
to one then if $\theta$ above is bounded below the $2\times 2$ mixed Hessian of the function $(x;y_1,y_1')\to \psi(x,y)+\psi(x,y')$ has nonvanishing determinant.  Thus, in
this case \eqref{3.12} just follows from H\"ormander's non-degenerate $L^2$-oscillatory integral lemma in \cite{H} (see \cite[Theorem 2.1.1]{SBook}).  Therefore,
it suffices to prove \eqref{3.12} when $\theta$ is bounded above by a fixed positive constant, and so Proposition~\ref{prop3.1} and hence
Theorem~\ref{theorem1.1} are a consequence of the following

\begin{lemma}\label{lemma3.2}
Suppose that $b\in C_0^\infty({\mathbb R}^2\times {\mathbb R}^2)$ vanishes when $|(s,t)|\ge \delta$.
%satisfies $b(x;s,t)=0$ when $x$ and $(s,t)$ are
%in the complement of a $\delta$-ball.  
Then if $\Psi \in C^\infty({\mathbb R}^2\times {\mathbb R}^2)$
is real and \eqref{3.20}--\eqref{3.23} are valid there is a uniform constant $C$ so that if $\delta>0$ and $\theta>0$ are smaller than  a fixed positive constant 
and
$$T_\la F(x)=\iint e^{i\la\Psi(x;s,t)}b(x;s,t) F(s,t)\, ds dt,$$
then we have
\begin{equation}\label{3.24}
\|T_\la F\|_{L^2({\mathbb R}^2)}\le C\la^{-1}\theta^{-\frac12}\|F\|_{L^2({\mathbb R}^2)}.
\end{equation}
\end{lemma}

We shall include the proof of this result for the sake of completeness even though it is a standard result.  It is  a slight variant of the main lemma in H\"ormander's
proof of the Carleson-Sj\"olin theorem in \cite{H} (see \cite[pp. 61-62]{SBook}).  H\"ormander's proof gives this result in the special case where $y_2=y'_2$, and, as above,
$\Psi$ is defined by two copies of the Riemannian distance function.  The case where $y_2$ and $y_2'$  are not equal to each other introduces some technicalities that, as we shall see, are straightforward
to overcome.

\begin{proof}  Inequality \eqref{3.24} is equivalent to the statement that $\|T^*_\la T_\la\|_{L^2\to L^2}\le C\la^{-2}\theta^{-1}$.  The kernel of
$T^*_\la T_\la$ is
\begin{multline*}K(s,t;s',t')=\iint e^{i\la(\Psi(x;s,t)-\Psi(x;s',t'))}a(x;s,t,s',t')\, dx_1dx_2,
\\ \text{if } \, \, a(x;s,t,s',t')=b(x,s,t)\overline{b(x;s',t')},
\end{multline*}
Therefore, we would have this estimate if we could show that
\begin{multline}\label{3.25}
|K(s,t;s',t')|
\le C\theta^{1-N}\bigl(1+\la|(s-s',t-t')|\bigr)^{-N}
\\
+
C\theta\bigl(1+\la\theta|(s-s',t-t')|\bigr)^{-N}, 
\quad
N=0,1,2,3,
\end{multline}
for then by using the $N=0$ bounds for  the regions where $|(s-s',t-t')|\le (\la\theta)^{-1}$ and the $N=3$ bounds in the complement, we see that
$$\sup_{s,t}\iint |K|\, ds'dt', \, \, \sup_{s',t'}\iint |K|\, ds dt \le C\la^{-2}\theta^{-1},$$
which means that, by Young's inequality, $\|T^*_\la T_\la\|_{L^2\to L^2}\le C\la^{-2}\theta^{-1}$, as desired.

The bound for $N=0$ follows from the first part of \eqref{3.20}.  To prove the bounds for $N=1,2,3$, we need to integrate by parts.

Let us first handle the case where
\begin{equation}\label{3.26}|s-s'|\ge A^{-1}|t-t'|,
\end{equation}
where $A\ge1$ is a possibly fairly large constant which we shall specify in the next step.
By the second part of \eqref{3.21} and by \eqref{3.23}, we conclude that if $\delta>0$ is sufficiently small (depending on $A$), we have
\begin{equation}\label{3.27}\Bigl|\frac\partial{\partial x_1}\bigl(\Psi(x;s,t)-\Psi(x;s',t')\bigr)\Bigr|\ge c|s-s'|, \quad
|s-s'|\ge A^{-1}|t-t'|,
\end{equation}
for some uniform constant $c>0$.  

Since $|K|$ is trivially bounded by the second term in the right side of \eqref{3.25} when $|s-s'|\le (\la\theta)^{-1}$ and
\eqref{3.26} is valid,
we shall assume that $|s-s'|\ge (\la\theta)^{-1}$.  

If we then write
\begin{multline}\label{i}
e^{i\la(\Psi(x;s,t)-\Psi(x;s',t'))}=Le^{i\la(\Psi(x;s,t)-\Psi(x;s',t'))}, 
\\ \text{where } \, L(x,D)=\frac1{i\la(\Psi'_{x_1}(x;s,t)-\Psi'_{x_1}(x;s',t'))}\, \frac\partial{\partial x_1},
\end{multline}
then we obtain
$$|K|\le \iint |(L^*(x,D))^N \, a(x;s,t,s',t')| \, dx.$$
Note that
\begin{multline}\label{3.28}
|\la(\Psi'_{x_1}(x;s,t)-\Psi'_{x_1}(x;s',t'))|^N \, |(L^*)^Na|
\\
\le C_N\sum_{0\le j+k\le N} \Bigl|\frac{\partial^j}{\partial x_1^j} a\Bigr| \times \sum_{\alpha_1+\dots+\alpha_k\le N}\frac{\prod_{m=1}^k
\bigl|\frac{\partial^{\alpha_m}}{\partial x_1^{\alpha_m}}
(\Psi'_{x_1}(x;s,t)-\Psi'_{x_1}(x;s',t'))\bigl|}{\bigl| \Psi'_{x_1}(x;s,t)-\Psi'_{x_1}(x;s',t')\bigr|^k}.
\end{multline}
Clearly,
\begin{equation}\label{j}\prod_{m=1}^k
\bigl|\frac{\partial^{\alpha_m}}{\partial x_1^{\alpha_m}}
(\Psi'_{x_1}(x;s,t)-\Psi'_{x_1}(x;s',t'))\bigl|\le C_k|(s-s',t-t')|^k,\end{equation}
and consequently, by \eqref{3.26} and \eqref{3.27},
\begin{equation}\label{k}
\frac{\prod_{m=1}^k\bigl|\frac{\partial^{\alpha_m}}{\partial x_1^{\alpha_m}}
(\Psi'_{x_1}(x;s,t)-\Psi'_{x_1}(x;s',t'))\bigl|}{\bigl|\Psi'_{x_1}(x;s,t)-\Psi'_{x_1}(x;s',t')\bigr|^k}\le C_{A,k}.
\end{equation}
Since by \eqref{3.20}, we have that $|\partial_{x_1}^j a|\le C\theta^{-j}$, $j=0,1,2,3$, and \eqref{3.20} also says that $a$ vanishes when $|x_1|$ is larger than
a fixed multiple of $\theta$, we conclude from \eqref{3.27}-\eqref{k} that if \eqref{3.26} holds then $|K|$ is dominated by the first term in the right side of \eqref{3.25}.

We now turn to the remaining case which is
\begin{equation}\label{l}
|t-t'|\ge A|s-s'|,
\end{equation}
and where the parameter $A\ge1$ will be specified.  By the first part of \eqref{3.21} and by \eqref{3.22} and
the fact that $|s|, \, |s'|, \, |t|, \, |t'|$ are bounded by a fixed multiple of $\theta$ in the support of $a$, it follows that we can fix $A$ (independent of $\theta$ small) so that if \eqref{l} is valid
then 
$$\bigl|\frac\partial{\partial x_2}\bigl(\Psi(x;s,t)-\Psi(x;s',t')\bigr)\bigr|\ge c\theta |t-t'|, \quad \text{on supp }a,
$$
for some uniform constant $c>0$.
Then since \eqref{3.17} implies that
$$\prod_{m=1}^k\Bigl| \frac{\partial^{\alpha_m}}{\partial x_2^{\alpha_m}}
\bigl(\Psi'_{x_2}(x;s,t)-\Psi'_{x_2}(x;s',t')\bigr)\Bigr|\le
C_k\theta^k|(s-s',t-t')|^k, \quad \text{on supp }a,
$$
and since, by \eqref{3.20}, 
$$|\partial^j_{x_2}a|\le C_N, \quad 1\le j\le N,$$
we conclude that, if we repeat the argument just given but now integrate
by parts with respect to $x_2$ instead of $x_1$, then $|K|$ is bounded by second term in the right side
of 
\eqref{3.25}, which completes the proof of Lemma~\ref{lemma3.2}.
\end{proof}

To conclude matters, we also need to prove the orthogonality estimates \eqref{3.4} and \eqref{3.10}.  Since \eqref{3.4} is a special case of \eqref{3.10}, we just need to establish the latter.

To see this, we note that by Lemma~\ref{kerlemma}, if $(\chi_\la Q^\mu_\theta)(x,y)$ denotes the kernel
of $\chi_\la Q^\mu_\theta$, then
\begin{multline*}
\bigl(\chi_\la Q^\mu_\theta\bigr)(x,y)\bigl(\chi_\la Q^{\mu'}_\theta\bigr)(x,y')
\overline{\bigl(\chi_\la Q^{\tilde\mu}_\theta\bigr)(x,\tilde y)}
\overline{\bigl(\chi_\la Q^{\tilde \mu'}_\theta\bigr)(x,\tilde y')}=O_N(\la^{-N})
\\
\text{if } \, \, x\notin 
{\mathcal T}_{C\theta}(\gamma_\mu)\cap {\mathcal T}_{C\theta}(\gamma_{\mu'})
\cap {\mathcal T}_{C\theta}(\gamma_{\tilde \mu})
\cap {\mathcal T}_{C\theta}(\gamma_{\tilde \mu'}),
\end{multline*}
with $C$ sufficiently large and the geodesics defined by \eqref{geod}.  On the other hand, if 
$x$ is in the above intersection of tubes, then the condition on $(\mu,\mu',\tilde \mu, \tilde \mu')$
in \eqref{3.10} ensures that if the constant $C$ there is large enough we have
\begin{multline*}
\bigl|\nabla_x\bigl(d_g(x,y)+d_g(x,y')-d_g(x,\tilde y)-d_g(x,\tilde y')\bigr)\bigr|\ge c_0\theta
\\
\text{if }\, \, y\in {\mathcal T}_{C\theta}(\gamma_\mu), \, 
y'\in {\mathcal T}_{C\theta}(\gamma_{\mu'}), \, 
\tilde y\in {\mathcal T}_{C\theta}(\gamma_{\tilde \mu}),
\, \, \text{and } \, \tilde y'\in {\mathcal T}_{C\theta}(\gamma_{\tilde \mu'}),
\end{multline*}
for some uniform $c_0>0$.  Thus, \eqref{3.10} follows from Lemma~\ref{kerlemma} and a simple
integration by parts argument since we are assuming that
$\theta\ge \theta_0=\la^{-\frac12+\e_0}$ with $\e_0>0$.

\newsection{Relationships with Zygmund's $L^4$-toral eigenfunction bounds}

Recall that for $\torus$ Zygmund~\cite{Zy} showed that if $e_\la$ is an eigenfunction on $\torus$, i.e.,
\begin{equation}\label{4.1}e_\la(x)=\sum_{\{\e\in \Zt: \, |\ell|=\la\}}a_\ell e^{ix\cdot \ell},
\end{equation}
then
$$\|e_\la\|_{L^4(\torus)}\le C,$$
for some uniform constant $C$.

As observed in \cite{BGT}, using well-known pointwise estimates in two-dimensions, one has
$$\sup_{\gamma\in \varPi} \int_{\gamma} |e_\la|^2\, ds =O_\e(\la^\e),$$
for all $\e>0$.  This of course implies that one also has
$$\sup_{\gamma\in \varPi} \int_{{\mathcal T}_{\la^{-\frac12}}(\gamma)} \, |e_\la|^2 dx = O_\e(\la^{-\frac12 +\e}),$$
for any $\e>0$.

Sarnak \cite{Sar} made an interesting observation that having $O(1)$ geodesic restriction bounds for $\torus$, is equivalent to the statement that there is a uniformly bounded number
of lattice points on arcs of $\la S^1$ of aperture $\la^{-\frac12}$.  
\footnote{Cilleruelo and C\'ordoba~\cite{CC} showed that this is the case for arcs of aperture $\la^{-\frac12-\delta}$ for
any $\delta>0$.}

Using \eqref{1.1} we can essentially recover Zygmund's bound and obtain $\|e_\la\|_{L^4(\torus)}=O_\e(\la^\e)$ for every $\e>0$.  (Of course this just follows from the pointwise estimate, but it shows how the method is natural too.)

If we could push the earlier results to include $\e_0=0$ and if we knew that there were uniformly bounded restriction bounds, then we would recover Zygmund's estimate.

\end{document}